\documentclass[leqno,12pt]{article}
\usepackage{amssymb,amsmath,amsfonts,amsthm}
\usepackage{authblk}
\usepackage{epsfig,color,mathrsfs}
\usepackage{graphicx}
\usepackage{amscd}
\usepackage{caption}
\usepackage{mathrsfs,bbm}
\usepackage{multicol}
\usepackage{color}
\usepackage{subfig}

\newcommand{\be}{\begin{equation}}
\newcommand{\ee}{\end{equation}}
\newcommand{\ben}{\begin{eqnarray*}}
\newcommand{\een}{\end{eqnarray*}}

\newcommand{\ve}{\varepsilon}
\newcommand{\vp}{\varphi}
\newcommand{\ds}{\displaystyle}

\newcommand{\intl}{\int\limits}

\newcommand{\R}{\mathbb R}

\allowdisplaybreaks

\newtheorem{theorem}{Theorem}[section]

\newtheorem{lemma}[theorem]{Lemma}

\definecolor{darkgreen}{rgb}{0.09, 0.45, 0.27}
\definecolor{debianred}{rgb}{0.84, 0.04, 0.33}

\numberwithin{equation}{section}

\begin{document}
\title{The Gierer-Meinhardt system in the entire space with non-local proliferation rates} 

\author[ ]{Marius Ghergu$^{1,2}$, Nikos I. Kavallaris$^3$ and Yasuhito Miyamoto$^4$}

\affil[ ]{$^1$School of Mathematics and Statistics}
\affil[ ]{University College Dublin }
\affil[ ]{Belfield Campus, Dublin 4, Ireland}
\affil[ ]{E-mail: {\tt marius.ghergu@ucd.ie}}
\affil[ ]{$^2$Institute of Mathematics Simion Stoilow of the Romanian Academy}
\affil[ ]{21 Calea Grivitei St., 010702 Bucharest, Romania}
\affil[ ]{}

\affil[ ]{$^3$Karlstad University, Faculty of Health, Science and Technology}
\affil[ ]{Department of Mathematics and Computer Science, Sweden}
\affil[ ]{E-mail: {\tt nikos.kavallaris@kau.se}}
\affil[ ]{}

\affil[ ]{$^4$Graduate School of Mathematical Sciences}
\affil[ ]{The University of Tokyo}
\affil[ ]{3-8-1 Komaba, Meguro-ku,
Tokyo 153-8914, Japan}
\affil[ ]{E-mail: {\tt miyamoto@ms.u-tokyo.ac.jp}}


\maketitle

\begin{abstract} 
In this work, we present a novel stationary Gierer-Meinhardt system incorporating non-local proliferation rates, defined as follows:
$$
\begin{cases}
\displaystyle   -\Delta u+\lambda u=\frac{J*u^p}{v^q}+\rho(x) &\quad\mbox{ in }\R^N\, , N\geq 1,\\[0.1in]
\displaystyle  -\Delta v+\mu v=\frac{J*u^m}{v^s} &\quad\mbox{ in }\R^N.\\[0.1in]
\end{cases}
$$
This system emerges in various contexts, such as biological morphogenesis, where two interacting chemicals, identified as an activator and an inhibitor, are described, and in ecological systems modeling the interaction between two species, classified as specialists and generalists. The non-local interspecies interactions are represented by the terms $J*u^p, J*u^m$ where the $*$-symbol denotes the convolution operation in $\R^N$ with a kernel $J\in C^1(\R^N\setminus\{0\})$.
 In the system, we assume that $0<\rho\in C^{0, \gamma}(\R^N)$ with  $\gamma\in (0,1)$, while the parameters satisfy $\lambda, \mu, q,m,s>0$ and $p>1$.  Under various integrability conditions on the kernel 
$J$, we establish the existence and non-existence of classical positive solutions in the function space $C^{2, \delta}_{loc}(\R^N).$ These results further highlight the influence of the non-local terms, particularly the proliferation rates,  in the proposed model.
\end{abstract}


\medskip



\section{Introduction}

As early as 1952,  Alan Turing, in his seminal paper \cite{T52}, pioneered the use of reaction–diffusion systems to model morphogenesis, the process responsible for the regeneration of tissue structures in hydra—an organism only a few millimeters long, composed of approximately 100,000 cells. Observations of morphogenesis in hydra suggested the involvement of two chemical substances (morphogens): a slowly diffusing (short-range) activator and a rapidly diffusing (long-range) inhibitor.

Turing was the first to demonstrate that while diffusion typically exerts a smoothing and homogenizing effect on individual chemicals, the interaction of two or more chemicals with differing diffusion rates could destabilize uniform steady states in reaction–diffusion systems. This instability leads to the emergence of nonhomogeneous patterns in the distribution of these reactants, a phenomenon now widely recognized as {\it diffusion-driven instability (DDI)} or {\it Turing instability}.

Building on Turing’s groundbreaking concept, Gierer and Meinhardt \cite{GM72} proposed, in 1972, the activator–inhibitor system that has since become known as the Gierer–Meinhardt system. This model was specifically designed to capture the dynamics of hydra regeneration within a defined domain $\Omega\subset \R^N, N\geq 1$ and is given below:
\begin{equation}\label{lGMs}
\begin{cases}
\displaystyle   u_t-D_u\Delta u=-\lambda u+k_u\frac{u^p}{v^q}+\rho(x), &\quad\mbox{ in }\Omega\times(0,\infty)\, ,\\[0.1in]
\displaystyle  v_t -D_v\Delta v=-\mu v+k_v\frac{u^m}{v^s}, &\quad\mbox{ in }\Omega\times(0,\infty),\\[0.1in]
\displaystyle\frac{\partial u}{\partial \eta}=\displaystyle\frac{\partial v}{\partial \eta}=0,  &\mbox{ on }\partial\Omega\times(0,\infty),\,\\[0.1in]
u(x,0)=u_0(x), \; v(x,0)=v_0(x), &\quad\mbox{ in }\Omega,
\end{cases}
\end{equation}
where $\eta$ denotes the unit outer normal vector to $\partial \Omega$ whilst $u$ and $v$ stand for the concentrations
of the activator and the inhibitor respectively. System \eqref{lGMs} intends to provide a thorough explanation of symmetry breaking as well as of {\it de novo pattern formation} by virtue of the coupling of a local activation and a long-range inhibition process.
The two subpopulations move randomly with diffusion rates $D_{u}$ and $D_{v}$ and usually the inhibitor population $v$ is  much faster than the activator population $u.$

The activator and inhibitor subpopulations die at rates $\lambda>0$ and $\mu>0$, and proliferate at rates $k_{u}$ and $k_{v}$, respectively. The inserted nonlinearities $\frac{u^p}{v^q}, \frac{u^m}{v^s}$ describe the fact that the activator promotes the differentiation process and it stimulates its own production, whereas the inhibitor acts a suppressant against the self-enhancing activator to prevent the unlimited growth. The  power coefficients $p,q,m,s$ in the proliferation terms model the strength of interactions between the two populations and satisfy the relations $p>1, q,m>0$ and $s>-1,$ cf. \cite{KS17, KS18}. Furthermore, the term $\rho(x)\geq 0$ is called {\it basic production term} for the activator subpopulation $u$ and represents the amount of activator produced by cell in a unit time. It is assumed that $\rho\in C^{0,\gamma}(\overline{\Omega})$, $\gamma\in (0,1)$. Regarding the initial data, it is usually considered that $u_0, v_0\in C^{2,\gamma}(\overline{\Omega}),\, u_0(x)>0, \; v_0(x)>0 $ in $\overline{\Omega}$ and $\frac{\partial u_0}{\partial \eta}|_{\partial \Omega}=\frac{\partial v_0}{\partial \eta}|_{\partial \Omega}=0$.

In the context of system \eqref{lGMs}, its dynamics can be effectively characterized by two key indices: the net self-activation index $\chi := \frac{p-1}{m}$ and the net cross-inhibition index $\delta := \frac{q}{s+1}$.

The index $\chi$ quantifies the interplay between the self-activation of the activator and the cross-activation of the inhibitor. A large value of $\chi$ implies that the activator experiences significant net growth regardless of the influence of the inhibitor. Conversely, the index $\delta$ measures the extent to which the inhibitor suppresses both the activator's production and its own. When $\delta$ is large, the inhibitor strongly suppresses the activator's production, leading to a notable reduction in its activity.

From a biological and mathematical perspective, it is typically assumed that the parameters $p,q,m$, and $s$ satisfy $\chi<\delta$ which leads to the condition
\begin{equation}\label{tc}
\sigma:=\frac{mq}{(p-1)(s+1)}>1,
\end{equation}
commonly referred to as the {\it Turing condition}. This condition ensures the emergence of Turing instability patterns, which play a crucial role in various natural and mathematical phenomena, as discussed in \cite{KS17, KS18, NST06}.  Under this setting, the existence of global solutions to \eqref{lGMs} was obtained in \cite{J06} (see also \cite{MT87, R84} for some partial results on the exponents $p,q,m,s$). The {\it anti-Turing condition} $\sigma\leq 1$ is also relevant and appears in some recent mathematical studies \cite{G23,G24,KS17, KS18, KBM21}. Nowadays, there is a vast literature around the Gierer-Meinhardt model and its steady-states in both bounded and unbounded domains, see \cite{DKW03, DKZ21, DGKZ22, G09, G23, G24, GM25, KW08, KWY13, M06, M07, NST06, WW99}.

In an ecological framework, system \eqref{lGMs} can be employed to model the coexistence of specialist and generalist species within a shared habitat—a phenomenon extensively documented in natural ecosystems, cf. \cite{Buchi2014}. 
Several ecological mechanisms have been proposed to explain this coexistence, including habitat selection \cite{Morris1996_CoexistenceHabitatSelection}, variations in dispersal rates \cite{Buchi2014}, and adaptive foraging behaviors in response to environmental fluctuations \cite{WilsonYoshimura1994}. Among these, studies focusing on coexistence through differences in foraging strategies often consider systems with two distinct resources. For example, the authors in \cite{WilsonYoshimura1994} investigated the dynamics of coexistence between generalist and specialist populations in an environment featuring dual resources. In  \cite{WilsonYoshimura1994} is highlighted the role of foraging strategies in shaping the interactions and stability of such populations.

In this context, the variable $u$ represents the specialist subpopulation, which is characterized by a lower food consumption and higher motility measured again by the diffusion coefficient $D_u$ and death rate equal to $\lambda.$   On the other hand, the variable $v$ corresponds to the generalist subpopulation, which requires a higher amount of food but has lower motility given by the diffusion coefficient $D_v,$ whilst it dies with rate equal to $\mu.$ 
The enhanced motility exhibited by specialists enables them to search for and acquire food resources with greater efficiency compared to generalists. This advantage in foraging provides them with a significant competitive edge. This behavior can be mathematically represented through the reaction terms $\frac{u^p}{v^q}, \frac{u^m}{v^s}$ in \eqref{lGMs}, which encapsulate the dynamics of resource acquisition and utilization.
Despite this, both subpopulations depend on similar herbivorous food resources, creating a competitive yet interdependent dynamic within the ecosystem. This modeling framework facilitates the analysis of their interactions and provides insights into the ecological mechanisms that sustain their coexistence.

In the context of morphogenesis, the condition  $D_u\ll D_v,$ which implies that inhibitors diffuse significantly faster than activators, is frequently emphasized. This disparity in mobility plays a critical role in pattern formation, as it satisfies the Turing conditions for instability in local systems like \eqref{GMs}, ultimately driving morphogenetic processes. Conversely, in ecological settings, the condition $D_u\gg D_v,$
  becomes more relevant, reflecting the greater mobility of specialists compared to generalists. However, recent studies (e.g., \cite{GGW20, SWC21}) suggest that while differing diffusion rates are crucial for Turing instability, other factors—such as reaction kinetics and specific parameter values—can also lead to pattern formation, even in systems where diffusion coefficients are comparable.

 Classical formulations of the Gierer-Meinhardt model, as \eqref{lGMs}, often assume local interactions, which may not fully represent the spatially distributed nature of biological processes such as morphogen signals or tissue growth. Introducing non-local proliferation rates for the activator enriches the model by accounting for spatial averaging effects or long-range interactions, which are critical in phenomena like stem cell regulation or tumor growth dynamics. This extension aligns with experimental observations that highlight the significance of non-local feedback mechanisms in driving emergent patterns in various biological systems (\cite{Mein82, Mur03}). By exploring these non-local effects, the model gains enhanced realism, offering novel insights into developmental processes and aiding in the design of bio-inspired systems. Resource  consumption patterns within populations can exhibit variation based on the motility of individuals. Specifically, less motile individuals typically engage in local consumption, where interactions and resource consumption occur primarily within a limited range. On the other hand, more motile individuals are characteristic of non-local consumption, where their interactions and resource consumption extend over larger spatial areas \cite{Atamas2012_Non-localCompetitionResources}. This non-local behavior has significant implications for competing populations, as it introduces long-range effects into the competitive dynamics. The impact of non-local resource consumption on population interactions has been further explored in works such as \cite{BaylissVolpert2015_CompetingPopulationsnon-local}, highlighting the complexity of ecological interactions when spatial structure and motility are considered. 

Motivated by the considerations outlined above, and under the assumption that the boundary effects of the interaction region have a minimal influence on the overall dynamics of the populations, we introduce the following non-local Gierer-Meinhardt system 
\begin{equation}\label{PGMs}
\begin{cases}
\displaystyle   u_t-D_u\Delta u+\lambda u=\frac{J * u^p}{v^q}+\rho(x) &\quad\mbox{ in }\R^N\times(0,\infty)\, , N\geq 1,\\[0.1in]
\displaystyle  v_t -D_v\Delta v+\mu v=\frac{J * u^m}{v^s} &\quad\mbox{ in }\R^N\times(0,\infty),\\[0.1in]
u(x,t), v(x,t)\to 0 &\mbox{ as }|x|\to \infty,\,\mbox{for any}\; t>0,
\end{cases}
\end{equation}
where $0<\rho\in C^{0,\gamma}(\R^N)$, $\gamma\in (0, 1)$  with $p>1,$ and $q,m,s,\lambda,\mu>0$.
This model builds upon the foundational framework of Gierer-Meinhart systems, incorporating non-local interactions to better capture the spatially extended influence of cell/population dynamics. In doing so, we account for the broader biological/ecological context, where interactions between cells/species are not confined to their immediate vicinity but instead extend over larger spatial scales, reflecting the influence of long-range biological/ecological effects. This extension allows for a more realistic representation of cells/species interactions, particularly when considering the motility and dispersal capabilities of individuals within the population.   The diffusion coefficients $D_u,D_v$ are assumed to be of comparable magnitude. For the sake of simplicity, throughout this work, they are further simplified by setting $D_u=D_v=1.$

It is worth noting that non-local interactions have already been incorporated into the Fisher-KPP equation \cite{L20} and various cancer models \cite{KLS23, L11, L13} to provide a more accurate representation of the underlying biological phenomena.

Turning back to the model \eqref{PGMs},  the interaction terms $\frac{J*u^p}{u^q}$ and $\frac{J*u^m}{u^s}$ represent the non-local effects, where the growth of each population depends not only on local interactions but also on the presence of the competing cells/species at distant locations. The star symbol in \eqref{PGMs} refers to the convolution operation in $\R^N$, that is,
$$
J*u^p(x)=\intl_{\R^N} J(x-z)u^p(z) dz
\quad
\quad\mbox{ for all }x\in \R^N,
$$
and $ J*u^m(x)$ is defined similarly. 
 This mechanism reflects the idea that more motile individuals, such as the specialists in the model, have an extended influence on the distribution and consumption of resources, affecting populations over a larger spatial domain, cf. \cite{BaylissVolpert2015_CompetingPopulationsnon-local}. In the context of morphogenesis, this non-local mechanism underscores the critical role of spatially distributed signal processes, where interactions extend beyond immediate neighbors, in driving the emergence of complex biological patterns and structures, cf. \cite{KS18}.

Moreover, the environmental heterogeneity introduced by $\rho(x),$ which represents spatially varying resource availability, further emphasizes the non-local nature of resource interactions. 


\section{Main results}
In the current work, we focus on the study of the steady-state problem of \eqref{PGMs}, given by the following system and in any spatial dimensions $N\geq 1$,
\begin{equation}\label{GMs}
\begin{cases}
\displaystyle   -\Delta u+\lambda u=\frac{J * u^p}{v^q}+\rho(x) &\quad\mbox{ in }\R^N\, , \\[0.1in]
\displaystyle   -\Delta v+\mu v=\frac{J * u^m}{v^s} &\quad\mbox{ in }\R^N,\\[0.1in]
u(x), v(x)\to 0 &\mbox{ as }|x|\to \infty,
\end{cases}
\end{equation}
where $\rho\in C^{0, \gamma}(\R^N)$, $\gamma\in (0,1)$ is positive  and the exponents $p>1, q,m,s>0$ are assumed to satisfy the {\it anti-Turing condition}
\[
\sigma:=\frac{mq}{(p-1)(s+1)}\leq 1.
\]
The convolution kernel  $J\in C^1(\R^N\setminus\{0\})$ is positive and there exist $c>0$ and $\theta\in (0, N)$ such that
\begin{equation}\label{J}
J(z)\leq c|z|^{\theta-N}\,,\quad |\nabla J(z)|\leq c|z|^{\theta-N-1} \quad\mbox{ for all }z\in \R^N\setminus\{0\}.
\end{equation}

We are looking for positive classical solutions of \eqref{GMs}, that is, functions $u, v\in C^2(\R^N)$ such that:
\begin{itemize}
\item $u, v>0$ in $\R^N$; 
\item for all $x\in \R^N$ we have
\end{itemize}
\begin{equation}\label{Jpm}
(J*u^p)(x)<\infty \quad\mbox{ and }\quad (J*u^m)(x)<\infty;
\end{equation}
\begin{itemize}
\item $u, v$ satisfy \eqref{GMs} pointwise.
\end{itemize}

The local version of \eqref{GMs}, in which $J*u^p$ and $J*u^m$ are replaced by $u^p$ and $u^m$ respectively, was discussed in \cite{G23} where it is shown that anti-Turing condition $\sigma<1$ is necessary for solutions to exhibit an exponential decay at infinity, see in particular \cite[Theorem 1.1]{G23}. Unlike the approach in \cite{G23}, which relies heavily on integral representation formulae for solutions of the Laplace and Schr\"odinger equation that hold in dimensions $N\geq 3$,  in the current work we resort to various integral estimates, motivated by the presence of the non-local convolution terms in \eqref{GMs}. This gives us the advantage to study the system \eqref{GMs} in all space dimensions $N\geq 1$. 

To analyze the problem \eqref{GMs}, we will initially focus on the scenario where  $J\in L^1(\R^N)$. Given that $u(x), v(x)\to 0$ as $|x|\to \infty,$   it becomes straightforward to verify that condition \eqref{Jpm} holds in this case.

\newpage

We assume that the kernel $J$ satisfies one of the following conditions:

\noindent $(E)\;$ ($J$ has exponential decay):
\smallskip

There exists $M_J>0$ such that for all $0<b<M_J$ we have 
\begin{equation}\label{E}
\intl_{\R^N}J(x-y)e^{-b|y|}dy\leq C(J,N,b) e^{-b|x|}\quad\mbox{ for all }x\in \R^N,
\end{equation}
\medskip

or

\medskip

\noindent $(P)\;$ ($J$ has power decay): 
\smallskip

For all $0<b<N$ we have 
\begin{equation}\label{P}
\intl_{\R^N}J(x-y)|y|^{-b} dy\leq C(J,N,b) (1+|x|)^{-b}\quad\mbox{ for all }x\in \R^N.
\end{equation}
As we shall see in Lemma \ref{l1} below, the kernel
$$
J(y)=(1+|y|)^{-\alpha}e^{-\beta|y|} \quad\mbox{ with }\; \alpha\geq 0\, , \beta>0,
$$
satisfies condition $(E)$ with $M_J=\beta$, while the kernel 
$$
J(y)=(1+|y|)^{-\alpha} \quad\mbox{ with }\; \alpha>N,
$$
satisfies condition $(P)$ but not condition $(E)$ above. 

 Throughout this paper, given two positive and continuous functions $f, g:\R^N\to (0, \infty)$, we use $f\simeq g$ to denote that the quotient $\frac{f}{g}$ is bounded in $\R^N$ between two positive constants.
Our first main result discusses the existence of a solution to \eqref{GMs} in the case where the convolution kernel has an exponential decay.
\begin{theorem}\label{ths}
Assume $0<\sigma\leq 1$ and $\rho\in C^{0,\gamma}(\R^N)$ satisfies 
\begin{equation}\label{ro}
\rho(x)\simeq (1+|x|)^{-a}e^{-b|x|}
\end{equation} 
for some $a\geq 0$ and $b>0$. Assume also that $J\in L^1(\R^N)$  satisfies conditions \eqref{J} and $(E)$ with 
\begin{equation}\label{mj}
M_J>(a+b)\max\{m, p\}.
\end{equation} 
Then, there exist $c,C>0$ depending on $N, m,p,q,s$, $\rho(x)$ and $J$ such that for all
\begin{equation}\label{lm}
\mu\geq C \quad \mbox{ and }\quad \lambda^{\frac{p(s+1)}{q}-m}\geq c\mu,
\end{equation}
the system \eqref{GMs} has a positive solution.  
\end{theorem}
We should note that a similar curve to \eqref{lm} was found in \cite[Theorem 1.1]{G23} for the local version of \eqref{GMs}. Since $0<\sigma\leq 1$, we see that the exponent of $\lambda$ in \eqref{lm} is positive:
$$
t:=\frac{p(s+1)}{q}-m=m\Big(\frac{1}{\sigma}-1\Big)+\frac{s+1}{q}>0.
$$
We state that the non-local effect of the convolution terms in \eqref{GMs} should be seen not in the exponent $t=\frac{p(s+1)}{q}-m$, but rather in the size of constants $C$, $c>0$ of \eqref{lm} which encompass the properties of the kernel $J$.

A similar result holds in the case where $J$ has a power decay. More precisely, we have:
\begin{theorem}\label{thp}
Assume that $0<\sigma\leq 1$,  $J\in L^1(\R^N)$  satisfies conditions \eqref{J} and  $(P)$. Assume also that
$\rho\in C^{0,\gamma}(\R^N)$ with
\begin{equation}\label{ro2}
\rho(x)\simeq (1+|x|)^{-a}\quad\mbox{ for some }\quad    \frac{\max\{\theta-1, 0\}}{\min\{m,p\}}<a< \frac{N}{\max\{m,p\}}, 
\end{equation} 
recalling that constant $\theta$ is given by  \eqref{J}.

Then, there exist $c,C>0$ depending on $N, m,p,q,s,\rho(x)$ and $J$ such that for all $\lambda$ and $\mu$ that satisfy \eqref{lm}, 
the system \eqref{GMs} has a positive solution.  
\end{theorem}
If $0<\theta\leq 1$, then the condition \eqref{ro2} on the exponent $a$ reads $0<a<N/\max\{m,p\}$.

In our next result we shall assume $J(y)=|y|^{\theta-N}$, for some $0<\theta<N$. Note that in this case \eqref{J} clearly holds but $J\not\in L^r(\R^N)$ for all $r\geq 1$. We have:

\begin{theorem}\label{thl}
Assume $J(y)=|y|^{\theta-N}$ with $0<\theta<N$ and that
$\rho\in C^{0,\gamma}(\R^N)$ with
\begin{equation}\label{ro3}
\rho(x)\simeq (1+|x|)^{-a}\quad\mbox{ for some }\quad  a> 0.
\end{equation} 
\begin{enumerate}
\item[{\rm (i)}] If $\displaystyle a\le\frac{\theta}{\min\{m,p\}}$, then \eqref{GMs} has no positive solutions;

\item[{\rm (ii)}] If $0<\sigma< 1$ and the following two conditions hold 
\begin{equation}\label{na1}
\frac{\theta}{\min\{ m, p\}}<a<\frac{N}{\max\{ m, p\}},
\end{equation}
\begin{equation}\label{na2}
a(p-1)(1-\sigma)\geq \theta\Big(1-\frac{q}{s+1}\Big),  
\end{equation}
then, there exist $c,C>0$ depending on $N, m,p,q,s, \theta$ and $\rho(x)$ such that if $\lambda$ and $\mu$ satisfy \eqref{lm}, 
then system \eqref{GMs} has a positive solution.  
\end{enumerate}
\end{theorem}
The effect of the non-local terms in \eqref{GMs} is more visible in Theorem \ref{thl}, where the kernel $J$ is not integrable in any Lebesgue space. Indeed, Theorem \ref{thl}(i) states that no positive solutions to \eqref{GMs} exist if the exponents $p$ and $m$ of the activator terms are  small. In turn, if these exponents are sufficiently large, then a positive solution exist provided \eqref{na1} and \eqref{na2} are satisfied. Let us note that condition \eqref{na1} is stronger than \eqref{ro2}. Also,  if $q\geq s+1$, then \eqref{na2} is automatically fulfilled.

Our proofs rely on the integral properties of the convolution $J*f$ as we describe in Section 3 combined with fixed point theorems and the features of the problem
\begin{equation}\label{www}
\begin{cases}
-\Delta w+\mu w=\frac{K(x)}{ w^{s}}, w>0 &\quad\mbox{ in }\R^N,\\
\;\;\; w(x)\to 0& \quad\mbox{ as }|x|\to \infty,
\end{cases}
\end{equation}
where $K\in C^{0, \gamma}_{loc}(\R^N)$ is positive and satisfies $K(x)\simeq (1+|x|)^{-a} e^{-b|x|}$ for some $a, b\geq 0$ and $s\geq 0.$

The remainder of this article is organized as follows. In Section 3 we collect some preliminary results.
First, we derive in Lemma \ref{l2} that under \eqref{J}, the convolution operator $\mathcal{T}f:=J*f$ maps $L^r(\R^N)$ into $C^{0, \nu}(\R^N)$ for some suitable $r>1$ and $\nu\in (0,1)$. This is a crucial point in our study of classical solutions to \eqref{GMs} and allows us to deduce the existence of positive solutions in the H\"older space $C^{2, \nu}_{loc}(\R^N)$ for some $\nu\in (0,1)$. The technical proof of Lemma \ref{l2} is postponed to the Appendix section. Another set of preliminary results discussed in Section 3 pertains to problem \eqref{www}. Next, in Sections 4, 5 and 6 we prove Theorems \ref{ths}, \ref{thp}  and \ref{thl}, respectively. Finally, we close this article with Section 7, which offers a discussion on our main results.

\section{Preliminary results}

\subsection{Properties of $J$-convolution}

This subsection is devoted to the study of the convolution operator  $J*f$. We start first with the following result that underlines the need for condition \eqref{J} in our study.  

\begin{lemma}\label{l2}
Assume $J\in C^1(\R^N\setminus\{0\})$ is positive and satisfies \eqref{J}.  Let $r> 1$ be such that $\theta-1<\frac{N}{r}<\theta$ and define
$$
\mathcal{T}f(x):=(J*f)(x)=\int_{\R^N} J(x-z)f(z) dz \quad \mbox{ for all }f\in L^r(\R^N). 
$$
Then, there exists $C=C(N, \theta, r)>0$ such that 
\begin{equation}\label{hol}
|\mathcal{T}f(x)-\mathcal{T}f(y)|\leq C|x-y|^{\theta-\frac{N}{r}}\|f\|_{L^r(\R^N)}\quad\mbox{ for all }x, y\in\R^N.
\end{equation}
\end{lemma}
The proof of Lemma \ref{l2} is rather technical and we shall postpone it to the Appendix section. 

Our next result provides some properties of kernels $J$ which satisfy condition $(E)$. 
\begin{lemma}\label{l1}
\begin{enumerate}
\item[\rm (i)] The kernel $J(y)=(1+|y|)^{-\alpha}e^{-\beta|y|}$, $\alpha\geq 0$, $\beta>0$  satisfies condition $(E)$ with $M_J=\beta$.
\item[\rm (ii)] Suppose the kernel $J$ satisfies condition $(E)$ and let $f\in L^1(\R^N)$ be such that 
$$
f(x)\simeq (1+|x|)^{-a}e^{-b |x|},
$$
for some $a,b\geq 0$ with $a+b<M_J$. Then, there exist two constants $C_0>c_0>0$ such that for all $x\in \R^N$ we have 
\begin{equation}\label{kernimp}
c_0 (1+|x|)^{-a}e^{-b |x|} \leq \intl_{|y|<1}  J(y) f(x-y) dy\leq (J*f)(x)\leq C_0 (1+|x|)^{-a}e^{-b |x|}.
\end{equation}
In particular,
\begin{equation}\label{kernj}
(J * f)(x) \simeq (1+|x|)^{-a}e^{-b |x|}.
\end{equation}
\end{enumerate}
\end{lemma}
\begin{proof}
(i) Take $x\in \R^N$ and $0<b<\beta$. Using $|y|\geq |x|-|x-y|$ we estimate
$$
\begin{aligned}
\int_{\R^N}J(x-y)e^{-b|y|}dy&= \int_{\R^N}(1+|x-y|)^{-\alpha}e^{-\beta|x-y|-b|y|}dy\\
&\leq  \int_{\R^N} e^{-\beta|x-y|-b|y|}dy\\
&\leq \int_{\R^N}e^{-\beta|x-y|-b(|x|-|x-y|)}dy\\
&= e^{-b|x|} \int_{\R^N}e^{-(\beta-b)|x-y|}dy\\
&= e^{-b|x|} \int_{\R^N}e^{-(\beta-b)|z|} dz\\
&= C e^{-b|x|}. 
\end{aligned}
$$
(ii) Take $x\in \R^N$ and let us fix $\delta>0$ such that $a+b<\delta< M_J$. We decompose
\begin{equation}\label{kernj1}
\int_{\R^N} J(x-y) (1+|y|)^{-a}e^{-b|y|} dy:=A_1+A_2,
\end{equation}
where
$$
\begin{aligned}
A_1&=\int_{|y|>|x|} J(x-y) (1+|y|)^{-a}e^{-b |y|} dy,\\[0.1in]
A_2&=\int_{|y|\leq |x|} J(x-y) (1+|y|)^{-a}e^{-b |y|} dy.
\end{aligned}
$$
To estimate $A_1$ we use the fact that $J\in L^1(\R^N)$ and we have 
\begin{equation}\label{kernj2}
\begin{aligned}
A_1&\leq (1+|x|)^{-a}e^{-b |x|}  \int_{|y|>|x|} J(x-y) dy\\
&\leq (1+|x|)^{-a}e^{-b |x|}  \int_{\R^N} J(z) dz\\
&=C (1+|x|)^{-a}e^{-b |x|}.
\end{aligned}
\end{equation}
To estimate $A_2$ we use the fact that $[0, \infty)\ni t\longmapsto (1+t)^{-a}e^{(\delta-b)t}$ is increasing (since $\alpha+\beta<\delta$) and we have 
$$
\begin{aligned}
A_2& = \int_{|y|\leq |x|} J(x-y) (1+|y|)^{-a}e^{(\delta-b)|y|} e^{-\delta |y|} dy\\
&\leq (1+|x|)^{-a}e^{(\delta-b)|x|}  \int_{|y|\leq |x|} J(x-y) e^{-\delta |y|} dy\\
&\leq (1+|x|)^{-a}e^{(\delta-b)|x|}  \int_{\R^N} J(x-y) e^{-\delta |y|} dy.
\end{aligned}
$$
Using the hypothesis \eqref{E} for $b=\delta$ we find
\begin{equation}\label{kernj3}
\begin{aligned}
A_2&  \leq (1+|x|)^{-a}e^{(\delta-b)|x|} e^{-\delta |x|} \\
&=C (1+|x|)^{-a}e^{-b |x|}.
\end{aligned}
\end{equation}
Using \eqref{kernj2} and \eqref{kernj3} into \eqref{kernj1} we deduce the upper bound in the estimate \eqref{kernimp}. For the lower bound we further estimate
$$
\begin{aligned}
\int_{\R^N} J(x-y) (1+|y|)^{-a}e^{-b|y|} dy & = \int_{\R^N} J(y) (1+|x-y|)^{-a}e^{-b |x-y|} dy \\
&\geq \int_{|y|<1} J(y)(1+|x-y|)^{-a}e^{-b|x-y|}dy\\
& \geq \int_{|y|<1} J(y) (1+|x|+|y|)^{-a}e^{-b(|x|+|y|)} dy \\
& \geq 2^{-a} (1+|x|)^{-a}e^{-b |x|} \int_{|y|<1} J(y)e^{-b|y|}  dy\\
&= C (1+|x|)^{-a}e^{-b |x|}.
\end{aligned}
$$
\end{proof}
For kernels $J$ which satisfy condition $(P)$ we have the following result. 
\begin{lemma}\label{l12}
\begin{enumerate}
\item[\rm (i)] The kernel $J(y)=(1+|y|)^{-\alpha}$, $\alpha>N$, satisfies condition $(P)$ but not condition $(E)$.
\item[\rm (ii)] Suppose the kernel $J$ satisfies condition $(P)$ and let $f\in L^1_{loc}(\R^N)$ be such that 
$$
f(x)\simeq (1+|x|)^{-a} \quad\mbox{ for some }0<a<N.
$$
Then, there exist two constants $C_0>c_0>0$ such that for all $x\in \R^N$ we have 
\begin{equation}\label{kernimp2}
c_0 (1+|x|)^{-a}\leq \intl_{|y|<1} J(y) f(x-y) dy\leq (J*f)(x)\leq C_0 (1+|x|)^{-a}.
\end{equation}
In particular,
\begin{equation}\label{kernj2a}
(J * f)(x) \simeq (1+|x|)^{-a}.
\end{equation}
\end{enumerate}
\end{lemma}
\begin{proof}
(i) Assume $J(y)=(1+|y|)^{-\alpha}$, $\alpha>N$, and let $0<b<N$. We decompose
$$
\intl_{\R^N}J(x-y)|y|^{-b}dy=\Big\{\intl_{|y|<\frac{|x|}{2} }+\intl_{\frac{|x|}{2}\leq |y|<2|x| } + \intl_{|y|\geq 2|x|}  \Big\} J(x-y)|y|^{-b}dy := B_1+B_2+B_3.
$$
To estimate $B_1$ we use 
$
|y|<\frac{|x|}{2}\Longrightarrow |x-y|\geq |x|-|y|\geq \frac{|x|}{2}$.
Thus,
\begin{equation}\label{b1}
\begin{aligned}
B_1&=\intl_{|y|<\frac{|x|}{2} } (1+|x-y|)^{-\alpha} |y|^{-b}dy\\
&\leq \Big(1+\frac{|x|}{2}\Big)^{-\alpha} \intl_{|y|<\frac{|x|}{2} }  |y|^{-b}dy\\
&\leq C(1+|x|)^{-\alpha}|x|^{N-b}\\
&\leq C|x|^{-b}\; , \quad \mbox{ since }\alpha>N.
\end{aligned}
\end{equation}

To estimate $B_2$ we use 
$|y|<2|x| \Longrightarrow |x-y|\leq 3|x|$ and then 
\begin{equation}\label{b2}
\begin{aligned}
B_2&\leq \intl_{|x-y|<3 |x|} (1+|x-y|)^{-\alpha} \Big(\frac{|x|}{2}\Big)^{-b}dy\\
&=C |x|^{-b} \intl_0^{3|x|} (1+t)^{-\alpha} t^{N-1}dt\\
&\leq C |x|^{-b}\, , \quad \mbox{ since $\alpha>N$}.
\end{aligned}
\end{equation}
Next, to estimate $B_3$ we use 
$|y|\geq 2|x| \Longrightarrow |x-y|\geq |y|-|x|\geq \frac{|y|}{2}$ and we have
\begin{equation}\label{b3}
\begin{aligned}
B_3&\leq \intl_{|y|\geq 2 |x|} \Big(1+\frac{|y|}{2}\Big)^{-\alpha} |y|^{-b } dy\\
&\leq (2|x|)^{-b}  \intl_{2|x| }^\infty  \Big(1+\frac{t}{2}\Big)^{-\alpha} t^{N-1}dt\\
&\leq C |x|^{-b}\, , \quad \mbox{ since $\alpha>N$}.
\end{aligned}
\end{equation}
Combining \eqref{b1}, \eqref{b2} and \eqref{b3} we deduce \eqref{P}.
However, $J$ cannot satisfy condition $(E)$ since for $|x|>1$ we have
$$
\begin{aligned}
\intl_{\R^N}J(x-y)e^{-b|y|}dy & \geq \intl_{|y|<1}(1+2|x|)^{-\alpha} e^{-b|y|}dy\\
&\geq (3|x|)^{-\alpha} \intl_{|y|<1}  e^{-b|y|}dy \geq c|x|^{-\alpha}.
\end{aligned}
$$
Thus, \eqref{E} cannot hold.

(ii) As in the proof of Lemma \ref{l1} we have
$$
\begin{aligned}
(J*f)(x)=\intl_{\R^N} J(y)f(x-y)d y & \geq C  \int_{|y|<1} J(y) (1+|x|+|y|)^{-a} dy \\
& \geq C \int_{|y|<1} J(y) (2+|x|)^{-a} dy \geq  C (1+|x|)^{-a}.
\end{aligned}
$$
If $|x|\leq 1$, then the upper bound in \eqref{kernimp2} holds since
$$
\begin{aligned}
(J*f)(x) & \leq C\intl_{\R^N} J(x-y)(1+|y|)^{-a} dy\\
& \leq C\intl_{\R^N}J(x-y)dy=C\intl_{\R^N}J(z) dz\\
&<C\leq 2^{a}C(1+|x|)^{-a}.
\end{aligned}
$$
To check that the upper bound in \eqref{kernimp2} holds for $|x|>1$, we decompose
\begin{equation}\label{d0}
\intl_{\R^N} J(x-y)f(y) dy=\Big\{ \intl_{|y|<|x|} +\intl_{|y|\geq |x|} \Big\} J(x-y)f(y) dy:=D_1+D_2.
\end{equation}
To estimate $D_1$ we let $0<a<b<N$ and since $|x|>1$ we have
$$
\begin{aligned}
D_1&\leq c \intl_{|y|<|x|} J(x-y)(1+|y|)^{-a} dy\\
&= c \intl_{|y|<|x|} J(x-y)(1+|y|)^{-b} (1+|y|)^{b-a} dy\\
&\leq c (1+|x|)^{b-a} \intl_{|y|<|x|} J(x-y) |y|^{-b} dy \quad \mbox{ since $b>a$}\\
&\leq c (1+|x|)^{b-a} |x|^{-b} \quad\mbox{ by property $(P)$}\\
&\leq c(1+|x|)^{-a} \, , \quad\mbox{ since } |x|>1.
\end{aligned}
$$
To estimate $D_2$ we also have
$$
\begin{aligned}
D_2&\leq c \intl_{|y|\geq |x|} J(x-y)(1+|y|)^{-a} dy\\
&= c (1+|x|)^{-a} \intl_{|y|\geq |x|} J(x-y)  dy\\
&\leq  c (1+|x|)^{-a} \intl_{\R^N} J(z)  dz \leq c(1+|x|)^{-a}.
\end{aligned}
$$
Using the last two estimates into \eqref{d0} we deduce the upper bound for $J*f$ in  \eqref{kernimp2}.
\end{proof}

The last result in this section provides an integral estimate obtained in \cite{GL25} and extends those in \cite{GKS20, GTbook}. This will be needed in the proof of Theorem \ref{thl}.
\begin{lemma}\label{lgl} {\rm (see \cite[Lemma 2.5]{GL25})}
Let $0<\theta<\kappa<N$. Then, there exist some constants $C>c>0$ such that for all $x\in \R^N$ we have 

\begin{equation}\label{st1}
c (1+|x|)^{\theta-\kappa}\leq \intl_{\R^N}\frac{dy}{|x-y|^{N-\theta}(1+|y|)^\kappa}\leq C (1+|x|)^{\theta-\kappa}.
\end{equation}
\end{lemma}

\subsection{The problem \eqref{www}}

This part is devoted to the study of the problem \eqref{www}. 
We first state the following comparison results, which are fairly standard. We include their proofs for the reader's convenience.
\begin{lemma}\label{comp0}
Let $\Omega\subset \R^N$ be an open and bounded set, $s\geq 0$, $\mu>0$, $K\in C(\overline\Omega)$ be a positive function and $w_1, w_2\in C^2(\overline\Omega)$ be such that 
$$
\begin{cases}
-\Delta w_1+\mu w_1-K(x)w_1^{-s}\geq 0  \geq -\Delta w_2+\mu w_2- K(x)w_2^{-s}\quad\mbox{ in }\Omega,\\[0.2cm]
\;\;\; w_1, w_2>0  \quad\mbox{ in }\Omega,\\[0.2cm]
\;\;\; \displaystyle  w_1\geq w_2 \quad\mbox{ on }\partial\Omega.
\end{cases}
$$
Then, $w_1\geq w_2$ in $\overline\Omega$.
\end{lemma}
\begin{proof}
Suppose by contradiction that the set $E=\{x\in \R^N:w_1(x)<w_2(x)\}$ is nonempty. Then $w_1-w_2$ has a minimum over $\overline\Omega$, say, at $x_0$, which is located in $E$. At this point, we have $\Delta (w_1-w_2)(x_0)\geq 0$. Using the conditions in the hypothesis we find 
$$
0> -\Delta (w_1-w_2)(x_0)+\mu (w_1-w_2) (x_0)\geq K(x_0)(w_1^{-s}(x_0)-w_2^{-s}(x_0))\geq 0,
$$
which is a contradiction. Hence, $E=\emptyset$ and thus $w_1\geq w_2$ in $\overline\Omega$. 
\end{proof}

\begin{lemma}\label{compp}
Let $s\geq 0$, $\mu>0$, $K\in C(\R^N)$ be a positive function and $w_1, w_2\in C^2(\R^N)$ be such that 
$$
\begin{cases}
-\Delta w_1+\mu w_1-K(x)w_1^{-s}\geq 0  \geq -\Delta w_2+\mu w_2- K(x)w_2^{-s}\quad\mbox{ in }\R^N,\\[0.2cm]
\;\;\; w_1, w_2>0  \quad\mbox{ in }\R^N,\\[0.2cm]
\;\;\; \displaystyle  \liminf_{|x|\to \infty} (w_1(x)-w_2(x)\big)\geq  0.
\end{cases}
$$
Then, $w_1\geq w_2$ in $\R^N$.
\end{lemma}
\begin{proof} Let $\varepsilon>0$. Since $\liminf_{|x|\to \infty} (w_1(x)-w_2(x))\geq 0$, there exists $R>0$ such that 
$w_1+\ve\geq w_2$ in $\R^N\setminus B_R$.
On the other hand, $w_1+\ve$ and $w_2$ satisfy the conditions of Lemma \ref{comp0} in $\Omega=B_R$. Applying this result it follows that $w_1+\ve \geq w_2$ in $B_R$. Hence, $w_1+\ve \geq w_2$ in $\R^N$. Since $\ve>0$ was arbitrarily taken, this entails $w_1\geq w_2$ in $\R^N$.
\end{proof}

Let $\phi(x)=(1+|x|^2)^{1/2}$, $x\in \R^N$, and for all $a, b\geq 0$ denote 
\begin{equation}\label{fab}
\Psi_{a,b}(x)=\phi(x)^{-a}e^{-b\phi(x)}\quad \mbox{ for all }x\in \R^N.
\end{equation} 

\begin{lemma}\label{calcule}
Let $a,b\geq 0$ and 
\begin{equation}\label{lam1}
\lambda>\max\{2a+2(a+b)^2, N(a+b)\}.
\end{equation}

Then,
\begin{equation}\label{cf}
\frac{\lambda}{2} \Psi_{a,b}\leq -\Delta \Psi_{a,b}+\lambda \Psi_{a,b}\leq 2\lambda \Psi_{a,b} \quad\mbox{ in }\R^N.
\end{equation}
\end{lemma}
\begin{proof} By direct calculations we have 
\begin{equation}\label{fabb}
\begin{aligned}
-&\Delta   \Psi_{a,b}+\lambda \Psi_{a,b}\\
&=\Big\{\lambda-b^2+\frac{b(N-2a-1)}{\phi}+\frac{b^2+a(N-a-2)}{\phi^2}+\frac{b(2a+1)}{\phi^3}+\frac{a(a+2)}{\phi^4}  \Big\}\Psi_{a,b}.
\end{aligned}
\end{equation}
Since $\phi\geq 1$ we have
$$
\begin{aligned}
-\Delta   \Psi_{a,b}+\lambda \Psi_{a,b} & \leq \Big\{\lambda-b^2+\frac{b(N-2a-1)+b(2a+1)}{\phi}+\frac{b^2+a(N-a-2)+a(a+2)}{\phi^2}  \Big\}\Psi_{a,b}\\
& =\Big\{\lambda-b^2+\frac{bN}{\phi}+\frac{b^2+aN}{\phi^2}  \Big\}\Psi_{a,b}\\
& \leq \Big\{\lambda-b^2+bN+b^2+aN \Big\}\Psi_{a,b}\\
& = \Big\{\lambda+N(a+b) \Big\}\Psi_{a,b}\\
& \leq 2\lambda\Psi_{a,b} \quad \mbox{ in }\R^N.
\end{aligned}
$$
This concludes the proof of the upper bound in \eqref{cf}. For the lower bound we  note that the quantity in \eqref{fabb} is increasing with respect to $N\geq 1$. Thus:
$$
\begin{aligned}
-\Delta   \Psi_{a,b}+\lambda \Psi_{a,b}& \geq \Big\{\lambda-b^2-\frac{2ab}{\phi}+\frac{b^2-a-a^2}{\phi^2}+\frac{b(2a+1)}{\phi^3}+\frac{a(a+2)}{\phi^4}  \Big\}\Psi_{a,b}\\
&\geq \Big\{\lambda-b^2-\frac{2ab}{\phi}-\frac{a+a^2}{\phi^2}\Big\}\Psi_{a,b}\\
&\geq \Big\{\lambda-b^2-2ab-a-a^2\Big\}\Psi_{a,b}\\
&= \Big\{\lambda-(a+b)^2-a \Big\}\Psi_{a,b}\\
&\geq \frac{\lambda}{2} \Psi_{a,b} \quad \mbox{ in }\R^N.
\end{aligned}
$$
\end{proof}

\begin{lemma}\label{exi}
Let $s\geq 0$, $\mu>0$ and $K\in C^{0, \gamma}_{loc}(\R^N)$ be a positive function such that $K(x)\simeq (1+|x|)^{-a}e^{-b|x|}$ for some $a,b\geq 0$. 
Then, for all
\begin{equation}\label{mmu}
\mu>\max\left\{\frac{2a}{1+s}+2\Big(\frac{a+b}{1+s}\Big)^2, N \Big(\frac{a+b}{1+s}\Big) \right\},
\end{equation}
the problem \eqref{www}
has a unique solution $w\in C^2(\R^N)$. Moreover, 
\begin{equation}\label{asm}
w(x)\simeq (1+|x|)^{-\frac{a}{1+s}} e^{-\frac{b}{1+s}|x|}.
\end{equation} 
\end{lemma}
\begin{proof}
By our assumption on $K$, there exist constants $c_2>c_1>0$ such that
$$
c_1\Psi_{a,b}\leq K\leq c_2 \Psi_{a,b}\quad\mbox{ in }\R^N.
$$
Define
$$
D=\left( \frac{2c_2}{\mu} \right)^{\frac{1}{1+s}}\quad\mbox{ and }\quad 
d=\left( \frac{c_1}{2\mu} \right)^{\frac{1}{1+s}}
$$
and set 
$$
\overline w=D\Psi_{\frac{a}{1+s},\frac{b}{1+s}}\, , \quad \underline w=d\Psi_{\frac{a}{1+s},\frac{b}{1+s}}.
$$
By the results in Lemma \ref{calcule} we have
$$
-\Delta \overline w+\mu \overline w\geq \frac{\mu}{2} \overline w= \frac{\mu}{2} D \Psi_{\frac{a}{1+s},\frac{b}{1+s}}=c_2 \Psi_{a,b} \Big( D \Psi_{\frac{a}{1+s},\frac{b}{1+s}}\Big)^{-s}\geq K(x) \overline w^{-s}\quad\mbox{ in }\R^N.
$$
This shows that $\overline w$ is a super-solution of \eqref{www}. Similarly, 
$$
-\Delta \underline w+\mu \underline w\leq 2\mu \underline w=2\mu d \Psi_{\frac{a}{1+s},\frac{b}{1+s}}=c_1 \Psi_{a,b} \Big( d \Psi_{\frac{a}{1+s},\frac{b}{1+s}}\Big)^{-s}\leq K(x) \underline w^{-s}\quad\mbox{ in }\R^N.
$$
Hence, $\underline w$ is a sub-solution of \eqref{www}.
For all $n\geq 1$ consider the problem
\begin{equation}\label{wwn}
\begin{cases}
-\Delta w+\mu w=K(x) w^{-s}, w>0 &\quad\mbox{ in }B_n,\\
\;\;\; w(x)=\overline w& \quad\mbox{ on }\partial B_n.
\end{cases}
\end{equation}
Then $\underline w$ and $\overline w$ are respectively sub and supersolutions to \eqref{wwn}. From the standard sub and super-solution method (see \cite[Section 3]{Pao92}), it follows that  \eqref{wwn} has a solution $w_n\in C^2(\overline B_n)$ and by the comparison principle in Lemma \ref{compp} we have $\underline w\leq w_{n+1}\leq w_n\leq \overline w$ in $B_n$. We may thus define $w:=\lim_{n\to \infty} w_n$ and have $\underline w\leq w \leq \overline w$ in $\R^N$. By standard elliptic arguments we deduce that $w\in C^2(\R^N)$ is a solution to \eqref{www}. The asymptotic behaviour \eqref{asm} follows from $\underline w\leq w \leq \overline w$ in $\R^N$. Finally, the uniqueness of $w$ as a solution to \eqref{www} is obtained by applying Lemma \ref{compp}. 
\end{proof}
\section{Proof of Theorem \ref{ths}}
From \eqref{ro} there exist $\beta>\alpha>0$ such that 
\begin{equation}\label{ro1}
\alpha \Psi_{a,b}(x)\leq \rho(x)\leq \beta \Psi_{a,b}(x)\quad\mbox{ for all }x\in \R^N,
\end{equation}
where $\Psi_{a,b}(x)$ is defined in \eqref{fab}. By \eqref{mj} and \eqref{kernimp} there exist $C_1>c_1>0$ and $C_2>c_2>0$ such that
\begin{equation}\label{mj1}
c_1\Psi_{ap,bp}(x) \leq J*\Psi_{ap,bp}(x)\leq C_1 \Psi_{ap,bp}(x) \quad\mbox{ for all }x\in \R^N,
\end{equation}
and
\begin{equation}\label{mj2}
c_2\Psi_{am,bm}(x) \leq J*\Psi_{am,bm}(x)\leq C_2 \Psi_{am,bm}(x) \quad\mbox{ for all }x\in \R^N.
\end{equation}
For $a\geq 0$, $b>0$ given by \eqref{ro}, we define
$$
A=\frac{am}{1+s}\quad\mbox{ and }\quad B=\frac{bm}{1+s}.
$$
For $D_1>d_1>0$ and $D_2>d_2>0$ which will be precised later we introduce
\begin{equation}\label{eqset}
\mathscr{X}=\left\{(u,v)\in C^{0,\gamma}_{loc}(\R^N)\times C^{0,\gamma}_{loc}(\R^N): 
\begin{aligned}
d_1 \Psi_{a,b} & \leq u\leq  D_1 \Psi_{a,b}\\
d_2 \Psi_{A,B} & \leq v\leq  D_2 \Psi_{A,B}
\end{aligned} \quad\mbox{ in }\R^N \right\}.
\end{equation}
Because of \eqref{lm}, we can assume that
$\lambda>0$ and $\mu>0$ satisfy \eqref{lam1} and \eqref{mmu}, respectively. 
For $(u,v)\in \mathscr{X}$, let $(Tu, Tv)\in C^{2, \gamma}_{loc}(\R^N)\times C^{2, \gamma}_{loc}(\R^N)$ be the unique solution of the de-coupled system
\begin{equation}\label{eqw}
\begin{cases}
\displaystyle   -\Delta Tu+\lambda Tu=\frac{J * u^p}{v^q}+\rho(x) &\mbox{ in }\R^N\, ,\\[0.1in]
\displaystyle   -\Delta Tv+\mu Tv=\frac{J * u^m}{(Tv)^s} &\mbox{ in }\R^N,\\[0.1in]
Tu, Tv\to 0\quad \mbox{ as } |x|\to \infty.&
\end{cases}
\end{equation}
We now set
\begin{equation}\label{fx}
\mathscr{F}:\mathscr{X}\to C^{2, \gamma}_{loc}(\R^N)\times C^{2, \gamma}_{loc}(\R^N),\quad \mathscr{F}(u,v)=(Tu, Tv).
\end{equation}
We claim first that $Tu$ and $Tv$ are well defined. Let us note that thanks to 
\eqref{ro} and \eqref{mj1} we have
\begin{equation}\label{wa0s}
\frac{J * u^p}{v^q}+\rho(x)\simeq \frac{\Psi_{ap, bp}}{\Psi_{Aq,Bq}}+\Psi_{a,b}\simeq \Psi_{ap-Aq, bp-Bq}+\Psi_{a,b}.
\end{equation}
Next we observe that since $0<\sigma\leq 1< p$ we have 
\begin{equation}\label{wa1s}
ap-Aq=ap-a\frac{mq}{s+1}=ap-a(p-1)\sigma\geq ap-a(p-1)=a
\end{equation}
and also
\begin{equation}\label{wa2s}
bp-Bq=bp-b\frac{mq}{s+1}=bp-b(p-1)\sigma\geq bp-b(p-1)=b.
\end{equation} 
Thus, due to \eqref{wa0s}-\eqref{wa2s} we obtain that
$$
\frac{J * u^p}{v^q}+\rho(x)\simeq \Psi_{a,b}(x).
$$
The existence and uniqueness of $Tu\in C^{2}(\R^N)$ follows from 
Lemma \ref{exi} with 
$$
\mu=\lambda, \quad K(x)=\frac{J * u^p}{v^q}+\rho(x)\simeq \Psi_{a,b}(x)\quad\mbox{  and }\quad s=0.
$$ 
Note that by the definition of $(u, v)\in \mathscr{X}$ we have $u^p\in L^r(\R^N)$ for all $r\geq 1$ so by Lemma \ref{l2} one has $K\in C^{0, \gamma}_{loc}(\R^N)$. Thus, by the standard Schauder estimates, one derives $Tu\in C^{2, \gamma}_{loc}(\R^N)$. 

The existence and uniqueness of $Tv\in C^{2}(\R^N)$ follow from
Lemma \ref{exi} with 
$$
K=J * u^m\simeq \Psi_{am,bm}.
$$ 
By Schauder estimates, $Tv\in C^{2,\gamma}_{loc}(\R^N)$. 
We should point out that by Lemma \ref{l2} one has $K=J * u^m\in C^{0, \nu}(\R^N)$ so the conditions in Lemma \ref{exi} are fulfilled.

\begin{lemma}\label{lconst1}
Let $C_1>c_1>0$ and $C_2>c_2>0$ be the constants appear in  \eqref{mj1} and \eqref{mj2}. 
Assume that 
\begin{subequations}
\begin{align}
& d_1  =\frac{\alpha}{2\lambda} \, , \label{d1}\\
& \frac{\lambda}{4} d_2^{q} = C_1 D_1^{p-1}  \quad\mbox{ and }\quad \frac{\lambda}{4} D_1\geq \beta \, , \label{d2} \\[0.07in]
& \frac{\mu}{2} D_2^{s+1}=C_2 D_1^m\, , \label{d3}\\[0.04in]
& 2\mu d_2^{s+1}  = c_2 d_1^m. \label{d4}
\end{align}
\end{subequations}
Then, with $\mathscr{F}$ defined in \eqref{fx}, we have  $\mathscr{F}(\mathscr{X})\subset \mathscr{X}$.
\end{lemma}
We shall prove that the four constants $D_1>d_1>0$ and $D_2>d_2>0$ that fulfill \eqref{d1}-\eqref{d4} indeed exist in Lemma \ref{lconst2} below. 
\begin{proof}
Let $(u,v)\in\mathscr{X}$. From \eqref{ro1}, \eqref{d1} and \eqref{cf} we have
$$
-\Delta Tu+\lambda Tu\geq \rho(x)\geq \alpha \Psi_{a,b}\geq -\Delta (d_1\Psi_{a,b})+\lambda (d_1\Psi_{a,b}) \quad\mbox{ in }\R^N.
$$
We apply Lemma \ref{compp} for $w_1=Tu$, $w_2=d_1\Psi_{a,b}$, $K(x)=\rho(x)$ and $s=0$ to deduce $Tu \geq d_1 \Psi_{a,b}$ in $\R^N$.

From \eqref{d2} we have 
\begin{equation}\label{dd}
\frac{C_1 D_1^p}{d_2^q}+\beta\leq \frac{\lambda}{4} D_1+\frac{\lambda}{4} D_1=\frac{\lambda}{2} D_1.
\end{equation}
Using this fact together with \eqref{ro1}, \eqref{mj1}, \eqref{eqset}, \eqref{eqw}, \eqref{wa1s}-\eqref{wa2s}, \eqref{dd} and \eqref{cf} we have

$$
\begin{aligned}
-\Delta Tu+\lambda Tu &=\frac{J*u^p}{v^q}+\rho(x) & \quad \mbox{ by \eqref{eqw} }\\[0.04in]
&\leq \frac{J*(D_1\Psi_{a,b})^p}{v^q}+\beta \Psi_{a,b} & \quad \mbox{ by \eqref{eqset}, \eqref{ro1} }\\[0.05in]
& \leq \frac{C_2D_1^p}{d_2^q} 
\frac{\Psi_{a,b}^p}{\Psi_{A,B}^q}+\beta \Psi_{a,b} & \quad \mbox{ by \eqref{ro1}, \eqref{mj1}}\\[0.07in]
& = \frac{C_2D_1^p}{d_2^q}  
\Psi_{ap-Aq,bp-Bq}+\beta \Psi_{a,b} & \\[0.07in]
& \leq \Big(\frac{C_2D_1^p}{d_2^q}+\beta\Big) \Psi_{a,b} & \quad \mbox{ by \eqref{wa1s}, \eqref{wa2s}}\\[0.07in]
& \leq \frac{\lambda}{2} D_1 \Psi_{a,b} &\quad \mbox{ by \eqref{dd}}\\[0.07in]
&\leq -\Delta (D_1\Psi_{a,b})+\lambda (D_1\Psi_{a,b}) &\quad \mbox{ by \eqref{cf}}.
\end{aligned}
$$
By Lemma \ref{compp} with $w_1=D_1 \Psi_{a,b}$, $w_2=Tu$, $K(x)=\frac{\lambda}{2} D_1 \Psi_{a,b}$ and $s=0$, we deduce $Tu \leq D_1 \Psi_{a,b}$ in $\R^N$.
So far, we have established the double inequality $d_1 \Psi_{a,b}\leq Tu\leq D_1 \Psi_{a,b}$ in $\R^N$.

By $u\geq d_1 \Psi_{a,b}$ and \eqref{mj2} we have
$$
-\Delta Tv+\mu Tv =\frac{J*u^m}{(Tv)^s}\geq \frac{J*(d_1\Psi_{a,b})^m}{(Tv)^s}\geq c_2d_1^m 
\frac{\Psi_{am,bm}}{(Tv)^s} \quad\mbox{ in }\R^N.
$$
On the other hand, by \eqref{cf} and \eqref{d4} we have
$$
-\Delta (d_2\Psi_{A,B})+\mu (d_2\Psi_{A,B}) \leq 2\mu d_2 \Psi_{A,B}= c_2d_1^m 
\frac{\Psi_{am,bm}}{(d_2 \Psi_{A,B})^s} \quad\mbox{ in }\R^N.
$$
We now apply Lemma \ref{compp} for $w_1=Tv$, $w_2=d_2 \Psi_{A,B}$, $K(x)=c_2d_1^m \Psi_{am,bm}$ to deduce $Tv\geq d_2 \Psi_{A,B}$ in $\R^N$.

Finally, from $u\leq D_1 \Psi_{a,b}$ and \eqref{mj2} we have
$$
-\Delta Tv+\mu Tv =\frac{J*u^m}{(Tv)^s}\leq \frac{J*(D_1\Psi_{a,b})^m}{(Tv)^s}
\leq C_2D_1^m 
\frac{\Psi_{am,bm}}{(Tv)^s} \quad\mbox{ in }\R^N.
$$
As above, by \eqref{cf} and \eqref{d3} we derive
$$
-\Delta (D_2\Psi_{A,B})+\mu (D_2\Psi_{A,B}) \geq \frac{\mu}{2} D_2 \Psi_{A,B}= C_2 D_1^m 
\frac{\Psi_{am,bm}}{(D_2 \Psi_{A,B})^s} \quad\mbox{ in }\R^N.
$$
By Lemma \ref{compp} for $w_1=D_2 \Psi_{A,B}$, $w_2=Tv$,  $K(x)=C_2 D_1^m \Psi_{am,bm}$ we deduce $Tv\leq D_2 \Psi_{A,B}$ in $\R^N$.

\end{proof}

\begin{lemma}\label{lconst2}
Let $C_1>c_1>0$ and $C_2>c_2>0$ be the constants appear in \eqref{mj1} and \eqref{mj2}. 
Then, there exist $D_1>d_1>0$ and $D_2>d_2>0$ depending on $\alpha,c_1,c_2, C_1, C_2$ such that for all $\lambda, \mu$ satisfying \eqref{lm}, conditions \eqref{d1}-\eqref{d4} hold.
\end{lemma}
\begin{proof}
From \eqref{d1} and \eqref{d4} we find 
\begin{equation}\label{d5}
d_2=\left(\frac{c_2\alpha^m}{2^{m+1}} \right)^{\frac{1}{1+s}}\lambda^{-\frac{m}{1+s}}\mu^{-\frac{1}{1+s}}.
\end{equation}
From \eqref{d2}$_1$ we find $D_1^{p-1}=\frac{\lambda}{4C_1}d_2^q$ and by \eqref{d5} we obtain
\begin{equation}\label{d6}
D_1=C_3 \lambda^{\frac{1}{p-1}-\sigma }\mu^{-\frac{\sigma}{m}}\quad\mbox{ where }\quad C_3=\left[\frac{1}{4C_1} \left(\frac{c_2\alpha^m}{2^{m+1}} \right)^{\frac{q}{1+s}} \right]^{\frac{1}{p-1}}.
\end{equation}
Using \eqref{d6} and \eqref{d3} we find 
\begin{equation}\label{d7}
D_2=C_4 \lambda^{\frac{m}{1+s}\big(\frac{1}{p-1}-\sigma\big) }\mu^{-\frac{\sigma+1}{1+s}}\quad\mbox{ where }\quad C_4=(2C_2C_3^m)^{\frac{1}{1+s}}.
\end{equation}
Now, if $\lambda$ and $\mu$ satisfy \eqref{lm}, for large constants $C,c>0$ in \eqref{lm} one has $D_1>d_1$, $D_2>d_2$ and \eqref{d2}$_2$. This concludes the proof.
\end{proof}
By elliptic arguments one can show that $\mathscr{F}$ defined in \eqref{fx} is continuous. 
At this stage we cannot apply the Schauder fixed point theorem as the set $\mathscr{X}$ is not compact. To overcome this difficulty, we shall work on approximating balls that exhaust $\R^N$.

For any $n\geq 1$ let
\begin{equation}\label{eqsetn}
\mathscr{X}_n=\left\{(u,v)\in C^{0,\gamma}_{loc}(\overline B_n) \times C^{0,\gamma}_{loc}(\overline B_n): 
\begin{aligned}
d_1 \Psi_{a,b} & \leq u\leq  D_1 \Psi_{a,b}\\
d_2 \Psi_{A,B} & \leq v\leq  D_2 \Psi_{A,B}
\end{aligned} \quad\mbox{ in } B_n \right\}.
\end{equation}
If $(u,v)\in \mathscr{X}_n$ we fix $(\widetilde u, \widetilde v)\in \mathscr{X}$ an extension of $(u,v)$ to $C^{0,\gamma}_{loc}(\R^N)\times C^{0,\gamma}_{loc}(\R^N)$. 
For any $(u,v)\in \mathscr{X}_n$ let 
$$
(T_nu, T_nv)\in C^{2,\gamma}_{loc}(\R^N)\times C^{2,\gamma}_{loc}(\R^N),
$$ 
be the solution of 
\begin{equation}\label{eqwn}
\begin{cases}
\displaystyle   -\Delta T_nu+\lambda T_nu=\frac{J * \widetilde u^p}{\widetilde v^q}+\rho(x) &\mbox{ in }\R^N\, ,\\[0.1in]
\displaystyle   -\Delta T_nv+\mu T_nv=\frac{J * \widetilde u^m}{(T_nv)^s} &\mbox{ in }\R^N,\\[0.1in]
T_nu, T_nv\to 0\quad \mbox{ as } |x|\to \infty.&
\end{cases}
\end{equation}
We now set
\begin{equation}\label{fxn}
\mathscr{F}_n:\mathscr{X}_n\to C^{2,\gamma}_{loc}(\overline B_n)\times C^{2,\gamma}_{loc}(\overline B_n),\quad \mathscr{F}_n(u,v)=(T_nu\mid_{\overline B_n}, T_nv\mid_{\overline B_n}).
\end{equation}
For $(u,v)\in \mathscr{X}_n$ we have $(\tilde u, \tilde v)\in \mathscr{X}$. Thus,  by \eqref{mj1} it follows that 
$$
(J * \widetilde u^p)(x)\leq (J * (D_1\Psi_{a,b})^p)(x)\leq C_1D_1^p\Psi_{ap, bp}(x)\quad\mbox{ for all }x\in \R^N,
$$
and 
$$
(J * \widetilde u^p)(x)\geq (J * (d_1\Psi_{a,b})^p)(x)\geq c_1d_1^p\Psi_{ap, bp}(x)\quad\mbox{ for all }x\in \R^N.
$$
Similar estimates hold for $J*\widetilde u^m$. 
With the same argument as in Lemma \ref{lconst1} we have that $\mathscr{X}_n$ is invariant for $\mathscr{F}_n$. We may now apply the Schauder fixed point theorem to deduce the existence of a fixed point $(u_n, v_n)\in \mathscr{X}_n$ of $\mathscr{F}_n$. Thus, from \eqref{eqwn} and \eqref{fxn} we have
\begin{equation}\label{eqwnn}
\begin{cases}
\displaystyle   -\Delta u_n+\lambda u_n=\frac{J * \widetilde u_n^p}{v_n^q}+\rho(x) &\mbox{ in }B_n\, ,\\[0.1in]
\displaystyle   -\Delta v_n+\mu v_n=\frac{J * \widetilde u_n^m}{v_n^s} &\mbox{ in }B_n. 
\end{cases}
\end{equation}
For convenience, let
$$
f_n:=\frac{J * \widetilde u_n^p}{v_n^q}+\rho(x)\,,\quad\mbox{ and }\quad 
g_n:=\frac{J * \widetilde u^m}{v_n^s}.
$$
Since $(u_n, v_n)\in \mathscr{X}_n$ and $(\widetilde u_n, \widetilde v_n)\in \mathscr{X}$, by \eqref{ro}, \eqref{wa1s}-\eqref{wa2s} we have
$$
\begin{aligned}
f_n & \leq \frac{J * \widetilde u_n^p}{v_n^q}+\beta \Psi_{a,b}\leq C\Psi_{a,b}\\[0.05in]
g_n& \leq  C \Psi_{A,B}
\end{aligned}
\quad\mbox{ in }\overline B_n,
$$
where $C>0$ is independent of $n$. Thus, $\{f_n\}_{n\geq 1}$ and $\{g_n\}_{n\geq 1}$ are bounded in $L^r(B_1)$ for all $r>1$. By standard elliptic regularity arguments, it follows that $\{u_n\}_{n\geq 1}$, $\{v_n\}_{n\geq 1}$ are bounded in $W^{2,r}(B_1)$. Since for $r>N$ and any smooth and bounded domain $\Omega\subset \R^N$ the embedding $W^{2,r}(\Omega) \hookrightarrow C^{1,\gamma}(\overline\Omega)$ is compact, it follows that there exists a subsequence $\{(u_n^1, v_n^1)\}_{n\geq 1}$ which converges in $C^{1, \gamma}(\overline B_1)$. 

By the same argument, $\{(u_n^1, v_n^1)\}_{n\geq 2}$ is bounded in $W^{2,r}(B_2)$, $r>N$. We thus deduce that there exists a subsequence $\{(u_n^2, v_n^2)\}_{n\geq 2}$  of $\{(u_n^1, v_n^1)\}_{n\geq 2}$  which converges in $C^{1,\gamma}(\overline B_2)$ 
Inductively, for any $k\geq 2$ we construct a subsequence $\{(u_n^k, v_n^k)\}_{n\geq k}$ of $\{(u_n^{k-1}, v_n^{k-1})\}_{n\geq k}$ which converges in $C^{1,\gamma}(\overline B_k)$. 
Let now
$$
U_n:=u_n^n\quad\mbox{ and }\quad  V_n:=v_n^n\quad\mbox{ for all }n\geq 1.
$$
Then $\{(U_n, V_n)\}_{n\geq 1}$ converges in $C^{1,\gamma}_{loc}(\R^N)\times C^{1,\gamma}_{loc}(\R^N)$ to some $(U,V)$ and from \eqref{eqwnn} we have
\begin{equation}\label{eqws}
\begin{cases}
\displaystyle   -\Delta U_n+\lambda U_n=\frac{J * \tilde U_n^p}{V_n^q}+\rho(x) &\mbox{ in }B_n\, ,\\[0.1in]
\displaystyle   -\Delta V_n+\mu V_n=\frac{J * \tilde U_n^m}{V_n^s} &\mbox{ in }B_n. 
\end{cases}
\end{equation}
We claim that $(U, V)\in C^{1,\gamma}(\R^N)\times C^{1,\gamma}(\R^N)$ is a weak solution of \eqref{GMs}. 
Indeed, let $\varphi \in C^2_c(\R^N)$. Test \eqref{eqws} with function $\vp$   we have
\begin{subequations}
\begin{align}
& \displaystyle   \intl_{\R^N} (\nabla U_n\nabla \vp+\lambda U_n\vp)dx=\intl_{\R^N}\intl_{\R^N} \frac{\vp(x)}{V_n(x)^q}  J(x-y) \widetilde U_n^p(y) dy dx +\intl_{\R^N}\rho \vp dx, \label{eqi3}\\[0.2in]
& \displaystyle   \int_{\R^N} (\nabla V_n\nabla \vp+\mu V_n\vp)dx=\intl_{\R^N}\intl_{\R^N} \frac{\vp(x)}{V_n(x)^s} J(x-y) \widetilde U_n^m(y) dy dx.\label{eqi4}
\end{align}
\end{subequations}
Using the convergence $(U_n, V_n)\to (U,V)$ in $C^{1,\gamma}_{loc}(\R^N)\times C^{1,\gamma}_{loc}(\R^N)$ we may pass to the limit with $n\to \infty$ in the left hand-side of \eqref{eqi3}-\eqref{eqi4}. 
Using the fact that $(\widetilde U_n, \widetilde V_n)\in \mathscr{X}$, $(\widetilde U_n, \widetilde V_n)\to (U, V)$ pointwise in $\R^N$ and the Lebesgue dominated theorem we may also pass to the limit in right hand-side of \eqref{eqi3}-\eqref{eqi4}. We thus obtain
$$
\begin{aligned}
& \displaystyle   \intl_{\R^N} (\nabla U\nabla \vp+\lambda U\vp)dx=\intl_{\R^N}\intl_{\R^N} \frac{\vp(x)}{V(x)^q}  J(x-y) U ^p(y) dy dx +\intl_{\R^N}\rho \vp dx, \\[0.2in]
& \displaystyle   \int_{\R^N} (\nabla V\nabla \vp+\mu V\vp)dx=\intl_{\R^N}\intl_{\R^N} \frac{\vp(x)}{V(x)^s} J(x-y) U^m(y) dy dx.
\end{aligned}
$$
Hence, $(U, V)\in C^{1,\gamma}_{loc}(\R^N)\times C^{1,\gamma}_{loc}(\R^N)$ is a positive weak solution of \eqref{GMs}. Now, by the standard regularity theory (see, for instance, \cite[Theorem 9.19]{GT2001}) it follows that $(U,V)\in C^{2, \gamma}_{loc}(\R^N)\times C^{2, \gamma}_{loc}(\R^N)$.
\qed

\section{Proof of Theorem \ref{thp}}

The proof of Theorem \ref{thp} is similar to that of Theorem \ref{ths} in which we take $b=0$. One main difference is that due to the power behavior of $\rho(x)$  in \eqref{ro2} we need to adjust the order of regularity in the relevant H\"older spaces.
Hereafter, we let 
\begin{equation}\label{fi}
\Phi_a(x):=\Psi_{a,0}(x)=(1+|x|^2)^{-a/2} \quad\mbox{ for all }x\in \R^N.
\end{equation} 
From \eqref{ro2} there exist  $\beta>\alpha>0$ such that 
\begin{equation}\label{ro12}
\alpha \Phi_{a}(x)\leq \rho(x)\leq \beta \Phi_{a}(x)\quad\mbox{ for all }x\in \R^N,
\end{equation}
where $\Phi_{a}(x)$ is defined in \eqref{fab}.
By \eqref{ro2} and \eqref{kernimp2} there exist constants $C_1>c_1>0$ and $C_2>c_2>0$ such that
\begin{equation}\label{mj12}
c_1\Phi_{ap}(x) \leq J*\Phi_{ap}(x)\leq C_1 \Phi_{ap}(x) \quad\mbox{ for all }x\in \R^N,
\end{equation}
and
\begin{equation}\label{mj22}
c_2\Phi_{am}(x) \leq J*\Phi_{am}(x)\leq C_2 \Phi_{am}(x) \quad\mbox{ for all }x\in \R^N.
\end{equation}
For $a> 0$ given by \eqref{ro2}, we define
$$
A=\frac{am}{1+s}.
$$
Let us rewrite the condition \eqref{ro2} as 
$$
\min\{am, ap\}>\max\{\theta-1, 0\}.
$$
Note also that 
$$N>\theta>\max\{\theta-1, 0\}.
$$
Thus, we can choose $r>1$ such that 
\begin{equation}\label{t1}
\min\Big\{ \theta, am, ap\Big\}>\frac{N}{r}>\max\{\theta-1, 0\}.  
\end{equation}
In particular, this yields
\begin{equation}\label{t2}
\theta-1<\frac{N}{r}<\theta \quad \text{ and }\quad \frac{N}{a\min\{m,p\}}<r.
\end{equation}
Let $\nu_0:=\theta-N/r$, so that $0<\nu_0<1$.
We define
\begin{equation}\label{nana}
\nu:=\min\{\nu_0,\gamma\}\in (0,1).
\end{equation}
Recall that $\rho\in C^{0,\gamma}(\R^N)$.

For $D_1>d_1>0$ and $D_2>d_2>0$ which will be precised later we introduce
\[
\mathscr{X}=\left\{(u,v)\in C^{0,\nu}_{loc}(\R^N)\times C^{0,\nu}_{loc}(\R^N): 
\begin{aligned}
d_1 \Phi_{a} & \leq u\leq  D_1 \Phi_{a}\\
d_2 \Phi_{A} & \leq v\leq  D_2 \Phi_{A}
\end{aligned} \quad\mbox{ in }\R^N \right\}.
\]
Because of \eqref{lm}, we can assume that $\lambda>0$ and $\mu>0$ satisfy \eqref{lam1} with $b=0$ and \eqref{mmu} with $b=0$, respectively. 
For $(u,v)\in \mathscr{X}$, let $Tu, Tv\in C^{2, \nu}_{loc}(\R^N)$ be the unique solutions of the de-coupled system
\[
\begin{cases}
\displaystyle   -\Delta Tu+\lambda Tu=\frac{J * u^p}{v^q}+\rho(x) &\mbox{ in }\R^N\, ,\\[0.1in]
\displaystyle   -\Delta Tv+\mu Tv=\frac{J * u^m}{(Tv)^s} &\mbox{ in }\R^N,\\[0.1in]
Tu, Tv\to 0\quad \mbox{ as } |x|\to \infty.&
\end{cases}
\]
We now set
\begin{equation}\label{fx2}
\mathscr{F}:\mathscr{X}\to C^{2,\nu}_{loc}(\R^N)\times C^{2,\nu}_{loc}(\R^N),\quad \mathscr{F}(u,v)=(Tu, Tv).
\end{equation}
We claim first that $Tu$ and $Tv$ are well defined. Let us note that thanks to 
\eqref{ro2} and \eqref{mj12} we have
\begin{equation}\label{wa0s2}
\frac{J * u^p}{v^q}+\rho(x)\simeq \frac{\Phi_{ap}}{\Phi_{Aq}}+\Phi_{a}\simeq \Phi_{ap-Aq}+\Phi_{a}.
\end{equation}
Next we notice that since $0<\sigma\leq 1< p$ we have 
\begin{equation}\label{wa1s2}
ap-Aq=ap-a\frac{mq}{s+1}=ap-a(p-1)\sigma\geq ap-a(p-1)=a.
\end{equation}
Thus, by \eqref{wa0s2} and \eqref{wa1s2} we find
$$
\frac{J * u^p}{v^q}+\rho(x)\simeq \Phi_{a}.
$$
The existence and uniqueness of $Tu\in C^{2}(\R^N)$ follows from 
Lemma \ref{exi} with 
$$
\mu=\lambda, \quad K(x)=\frac{J * u^p}{v^q}+\rho\simeq \Phi_{a},\quad\mbox{  and }\quad s=0.
$$ 
Note that from the definition of $(u, v)\in \mathscr{X}$
we have $u^p\in L^r(\R^N)$, since by \eqref{t1} we have $r>N/ap$.
By Lemma~\ref{l2} and \eqref{t2} one has $J * u^p\in C^{0,\nu}(\R^N)$.
Since $J*u^p/v^q\in C^{0,\nu}_{loc}(\R^N)$ and $\rho\in C^{0,\gamma}(\R^N)$, one has $K\in C^{0,\nu}_{loc}(\R^N)$.
We point out that the assumptions of Lemma~\ref{exi} are satisfied and the above argument works.
By the standard Schauder estimates, one derives $Tu\in C^{2, \nu}_{loc}(\R^N)$. 

The existence and uniqueness of $Tv\in C^{2}(\R^N)$ follow by
Lemma \ref{exi} and \eqref{mj22} with 
$$
K=J * u^m\simeq \Phi_{am}.
$$ 
We have $u^m\in L^r(\R^N)$, since by \eqref{t2} one has $r>N/am$.
By Lemma~\ref{l2} and \eqref{t2} it follows that $K\in C^{0,\nu}_{loc}(\R^N)$.
Thus, the assumptions of Lemma~\ref{exi} are satisfied.
Since $K(x)/(Tv)^s\in C^{0,\nu}_{loc}(\R^N)$, by the local Schauder estimates one derives $Tv\in C^{2,\nu}_{loc}(\R^N)$.

The proofs of Lemmas~\ref{lconst1} and \ref{lconst2} are also valid for $b=0$ and thus, we may follow the lines of the proof of Theorem~\ref{thp} (in which we take $b=0$). 
An important point to note is that the sequences ${f_n}$ and ${g_n}$, as defined in Section~4, are bounded in $L^r(\mathbb{R}^N)$ only for sufficiently large values of $r > 1$.
However, this difference does not affect the subsequent argument; we omit the details here.
This concludes the proof of Theorem~\ref{thp}.
\qed

\section{Proof of Theorem \ref{thl}}

(i) Suppose by contradiction that $\min\{ m, p\}\leq\frac{\theta}{a}$ and there exists a positive solution $(u, v)$ to \eqref{GMs}. Without losing the generality, we may assume $m\leq \frac{\theta}{a}$. 

Let $w=\Phi_a$, where $\Phi_a$ is defined in \eqref{fi}. Then, by \eqref{ro3} and the calculations \eqref{fabb} in Lemma \ref{cf} we have
$$
-\Delta w+\lambda w\leq C \rho(x)\quad\mbox{ in }\R^N,
$$
for some $C>0$. Then, 
$$
-\Delta u+\lambda u\geq \rho(x)\geq -\Delta(\frac{w}{C})+\lambda \frac{w}{C} \quad\mbox{ in }\R^N.
$$
Using Lemma \ref{compp} with $\mu=\lambda$, $s=0$ and $K=\rho(x)$ we deduce 
$$
u(x)\geq \frac{w(x)}{C}=\frac{1}{C}\Phi_a(x)\geq c|x|^{-a}\quad\mbox{in } \{|x|>1\}, \mbox{ for some }c>0.
$$
Now, for $|x|>1$ we find

$$
\begin{aligned}
(J*u^m)(x)& =\intl_{\R^N} |x-y|^{\theta-N}u(y)^m dy\\
&\geq C \intl_{|y|>2|x|} |x-y|^{\theta-N}|y|^{am} dy\\
&\geq C\intl_{|y|>2|x|} |y|^{\theta-N+am} dy=\infty,
\end{aligned}
$$
and this contradicts the condition \eqref{Jpm}.

(ii) Assume that $0<\sigma<1$ and that \eqref{na1}-\eqref{na2} hold. The existence of a positive solution to \eqref{GMs} follows the lines of the proof in Theorem \ref{thp}. Note first that \eqref{ro12} still holds. In light of Lemma \ref{lgl}, the estimates \eqref{mj12}-\eqref{mj22} are now replaced with 
\begin{equation}\label{mja}
c_1\Phi_{ap-\theta}(x) \leq J*\Phi_{ap}(x)\leq C_1 \Phi_{ap-\theta}(x) \quad\mbox{ for all }x\in \R^N,
\end{equation}
and
\begin{equation}\label{mjb}
c_2\Phi_{am-\theta}(x) \leq J*\Phi_{am}(x)\leq C_2 \Phi_{am-\theta}(x) \quad\mbox{ for all }x\in \R^N,
\end{equation}
where $C_1>c_1>0$ and $C_2>c_2>0$ are constants.
Define
$$
A=\frac{am-\theta}{1+s}.
$$
Fix now $r>1$ such that $\theta-1<\frac{N}{r}<\theta$. Note that by \eqref{na1} we have $a\min\{p,m\}>\theta$. Thus, \eqref{t1}-\eqref{t2} hold. We define $\nu$ as in \eqref{nana} and work now with the same space $\mathscr{X}$ and the same mapping $\mathcal{F}$ as in the proof of Theorem \ref{thp}. 
To prove that $\mathcal{F}$ is well defined, we note that 
$$
\frac{J * u^p}{v^q}+\rho(x)\simeq \frac{\Phi_{ap-\theta}}{\Phi_{Aq}}+\Phi_{a}\simeq \Phi_{ap-Aq-\theta}+\Phi_{a}.
$$
Next by \eqref{na2} we have 
$$
\begin{aligned}
ap-Aq-\theta & =ap-\frac{am-\theta}{s+1}q-\theta\\
&=ap-\frac{amq}{s+1}-\theta\Big(1-\frac{q}{s+1}\Big)\\
&=ap-a(p-1)\sigma -\theta\Big(1-\frac{q}{s+1}\Big)\\
&=a+a(p-1)(1-\sigma)-\theta\Big(1-\frac{q}{s+1}\Big)\\
&\geq a.
\end{aligned}
$$
Thus, by the above estimates, we deduce
$$
\frac{J * u^p}{v^q}+\rho(x)\simeq \Phi_{a}.
$$
Also, $J*u^m\simeq \Phi_{am-\theta}$. From now on one follows line by line the proof of Theorem~\ref{thp}. \qed

\section{Discussion}

In this paper we studied the positive steady-state solutions to a reaction-difusion system that originates from the model of Gierer and Meinhardt \cite{GM72} in 1972. Our goal was to investigate the influence of the non-local terms $J*u^p$ and $J*u^m$ in this model, which are included to better illustrate the biological phenomena.  

In the frame of the anti-Turing condition $0<\sigma\leq 1$ we provided a general setting that contains a broad classes of kernels $J\in C^1(\R^N\setminus\{0\})$ such as  
 $J(y)=(1+|y|)^{-\alpha}$, $J(y)=(1+|y|)^{-\alpha}e^{-\beta|y|}$ and $J(y)=|y|^{-\alpha}$. Our findings reveal that the various integrability conditions imposed on $J$ impact both the range of exponents $p,q,m, s$ as well as the death rates $\lambda$ and $\mu$ for which a solution exists. This is more visible if we assume a low order of integrability of the kernel $J$ as it was noted in Theorem \ref{thl} where $J\not\in L^r(\R^N)$ for all $r\geq 1$. In all situations discussed in the present article, the range of death rates $\lambda$ and $\mu$ is given by the curve \eqref{lm}; this  was also observed in the local case discussed in \cite{G23}. We pointed out that each situation discussed in Theorems \ref{ths}, \ref{thp} and \ref{thl} yields different values of the constants $C, c>0$ in \eqref{lm}. Moreover, these variations can be compared with the findings in \cite{G23}, which provide an alternative perspective on how the non-local terms influence the existence of solutions of the non-local problem \eqref{GMs}.

\section*{Appendix: Proof of Lemma \ref{l2}}

In this section, we prove Lemma \ref{l2} by providing all details of the estimate \eqref{hol}.

Let $x, y\in \R^N$, $x\neq y$. For $\zeta\in \R^N$ we decompose
$$
\mathcal{T}f(\zeta)=\intl_{\R^N}J(z-\zeta)f(z) dz :=\mathcal{T}_1(\zeta)+\mathcal{T}_2(\zeta),
$$
where
$$
\begin{cases}
\ds\mathcal{T}_1(\zeta)=\intl_{|z-x|<2|x-y|}J(z-\zeta)f(z) dz\\[0.4in]
\ds \mathcal{T}_2(\zeta)=\intl_{|z-x|\geq 2|x-y|} J(z-\zeta)f(z) dz
\end{cases}
\mbox{ for all }\zeta\in \R^N.
$$
Using \eqref{J} and H\"older's inequality we estimate
$$
\begin{aligned}
\mathcal{T}_1(x) &=\intl_{|z-x|<2|x-y|}J(z-x)f(z) dz\\
&\leq C \intl_{|z-x|<2|x-y|} |z-x|^{\theta-N} f(z) dz\\
&\leq C \biggl(\intl_{|z-x|<2|x-y|}\big|z-x\big|^{(\theta-N)\frac{r}{r-1}} dz\biggl)^{\frac{r-1}{r}}  \biggl(\intl_{|z-x|<2|x-y|}|f(y)|^r dy\biggl)^{\frac{1}{r}} \\[0.1cm]
&= C\biggl(\intl_{0}^{2|x-y|}t^{(\theta-N) \frac{r}{r-1}}  t^{N-1} dt\biggl)^{\frac{r-1}{r}}  \|f\|_r.\\
\end{aligned}
$$
Since $(\theta-N)\frac{r}{r-1}+N >0$ we deduce
\begin{equation}\label{e1}
\mathcal{T}_1(x)\leq C(N, \theta, r) \big|x-y\big|^{\theta-\frac{N}{r}}\|f\|_r.
\end{equation}
We proceed similarly to deduce
\begin{equation}\label{e2}
\mathcal{T}_1(y)\leq C\big|x-y\big|^{\theta-\frac{N}{r}}\|f\|_r.
\end{equation}
Next, we observe that
\begin{equation}\label{e3}
|\mathcal{T}_2(x)-\mathcal{T}_2(y)|\leq \intl_{|z-x|\geq 2|x-y|} \big|J(z-x)-J(z-y)\big|f(z) dz.
\end{equation}
We apply the Mean Value Theorem to $g(\zeta)=J(z-\zeta)$ over the segment $[x,y]\subset \R^N$. Combining this fact with \eqref{J} we find 
\begin{equation}\label{e4}
\begin{aligned}
\big|J(z-x)-J(z-y)\big|&=|g(x)-g(y)|\\
&=|\nabla g(\xi)||x-y|\quad\mbox{ for some }\xi \in [x,y]\\
&=|\nabla J(z-\xi)||x-y|\\
&\leq C|z-\xi|^{\theta-N-1}|x-y|.
\end{aligned}
\end{equation}
From $|z-x|\geq 2|x-y|$ and $\xi \in [x,y]$ we find
$$
|z-\xi|\geq |z-x|-|x-\xi|\geq |z-x|-|x-y| \geq \frac{|z-x|}{2}.
$$
Using this last estimate in \eqref{e4} and then in \eqref{e3} we deduce
$$
|\mathcal{T}_2(x)-\mathcal{T}_2(y)|\leq C |x-y| \intl_{|z-x|\geq 2|x-y|} \big|z-x\big|^{\theta-N-1}f(z) dz.
$$
Next, by H\"older inequality we further estimate
$$
\begin{aligned}
|\mathcal{T}_2(x)-\mathcal{T}_2(y)|& \leq C |x-y| \intl_{|z-x|\geq 2|x-y|} \big|z-x\big|^{\theta-N-1}f(z) dz \\[0.1cm]
&\leq C|x-y|\biggl(\intl_{|z-x|\geq 2|x-y|}\big|x-z\big|^{(\theta-N-1)\frac{r}{r-1}} dz\biggl)^{\frac{r-1}{r}}  \|f\|_r\\
&= C|x-y|\biggl(\intl_{2|x-y|}^\infty t^{(\theta-N-1)\frac{r}{r-1}} t^{N-1} dt\biggl)^{\frac{r-1}{r}}  \|f\|_r.\\
\end{aligned}
$$
Since $(\theta-N-1)\frac{r}{r-1}+N <0$ we deduce
\begin{equation}\label{e5}
|\mathcal{T}_2(x)-\mathcal{T}_2(y)| \leq C(N,\theta, q) \big|x-y\big|^{\theta-\frac{N}{r}}\|f\|_r.
\end{equation}
Finally, from \eqref{e1}, \eqref{e2} and \eqref{e5} we have
$$
\begin{aligned}
|\mathcal{T}f(x)-\mathcal{T}f(y)|& \leq \mathcal{T}_1(x)+\mathcal{T}_1(y)+|\mathcal{T}_2(x)-\mathcal{T}_2(y)|\\
&\leq  C(N,\theta, r) \big|x-y\big|^{\theta-\frac{N}{r}}\|f\|_r,
\end{aligned}
$$
which concludes the proof of Lemma \ref{l2}.

\end{document}